\documentclass[a4paper,reqno,12pt]{amsart}

\usepackage{graphicx}
\usepackage{mathrsfs}
\usepackage{amsfonts}
\usepackage{amssymb}
\usepackage{amsmath}
\usepackage{amsthm}
\usepackage{amscd}
\usepackage[all,2cell]{xy}
\usepackage{color}
\usepackage[pagebackref,colorlinks]{hyperref}
\usepackage{enumerate}
\usepackage[normalem]{ulem}
\usepackage{mathrsfs}
\usepackage{lineno}
\usepackage{bm}
\usepackage{bbm}

\usepackage{comment}

\numberwithin{equation}{section}

\usepackage{comment}

\theoremstyle{plain}
\newtheorem{thm}{Theorem}[section]

\newtheorem{lemma}[thm]{Lemma}
\newtheorem{conjecture}[thm]{Conjecture}
\newtheorem{question}[thm]{Question}

\theoremstyle{definition}
\newtheorem{deff}[thm]{Definition}
\newtheorem{example}[thm]{Example}
\theoremstyle{remark}
\newtheorem{rmk}[thm]{\bf Remark}

\newcommand{\K}{\mathsf{k}}

\def\M{\mathbb{M}}

\def\-{\text{-}}

\newcommand{\gr}{\operatorname{gr}}

\begin{document}

\title[Distinguishing  Leavitt algebras among Leavitt path algebras]{Distinguishing  Leavitt algebras among Leavitt path algebras of finite graphs by Serre property}

\author{Roozbeh Hazrat}
\address{Centre for Research in Mathematics and Data Science,
Western Sydney University,
Australia}\email{r.hazrat@westernsydney.edu.au}

\author{Kulumani M. Rangaswamy}
\address{Department of Mathematics\\University of Colorado\\Colorado Springs, CO 80918, USA}
\email{krangasw@uccs.edu}

\subjclass[2010]{}


\date{\today}

\begin{abstract} 
Two unanswered questions in the heart of the theory of Leavitt path algebras are whether Grothendieck group $K_0$ is a complete invariant for the class of unital 
purely infinite simple algebras and, a weaker question, whether $L_2$ (the Leavitt path algebra associated to a vertex with two loops) and its Cuntz splice algebra $L_{2-}$ are isomorphic. A positive answer to the first question implies the latter. In this short paper, we raise and investigate another question, the so-called Serre's conjecture, which sits in between of the above two questions: A positive answer to the classification question implies Serre's conjecture which in turn implies $L_2 \cong L_{2-}$. Along the way, we give new easy to construct algebras having stably free but not free modules. 

\end{abstract}

\maketitle

\section{Introduction}

In his article ``Faisceaux Algebriques Coherents'' published in Annals of Maths, 1955, J.P. Serre asked whether  
finitely generated projective modules over the polynomial ring $A=\K[x_{1},\cdots ,x_{n}]$, with coefficients in a field $\K$, are free. This became known as \emph{Serre's conjecture}. Serre's conjecture induced intensive research activity in algebraic geometry. There is even a 440 page book by T.Y. Lam~\cite{lam} devoted entirely to the origins and implications of this conjecture. The fact that projective modules over a field $\K$, the polynomial ring $\K[x]$ and the Laurent ring $\K[x,x^{-1}]$ are free follows from the classical result that projective modules over PID are free. 

The second named author learned about Serre's conjecture around the time that he finished his undergraduate degree and his mentor C. Seshadri proved the first important result in this direction in 1958, that Serre's conjecture holds for the polynomial ring $\K[x_1,x_2]$. Many attempts ensued. Finally in 1976, Daniel Quillen in USA and, independently, Andrei Suslin in the Soviet Union showed that Serre's conjecture holds for any $n\geq 1$~\cite{lam}.

The question is also closely related to rings which have stably free (projective) modules which are not free. Constructing rings with stably free modules which are not free is not immediate. There are very interesting constructions by Kaplansky (in the case of commutative rings) and 
Ojanguren and Sridharan in the non-commutiave case, which also show that polynomial ring $D[x_1,\dots,x_n]$, with coefficients in a (non-commutative) division ring $D$, has stably free projective modules which are not free (in contrast to Quillen-Suslin's theorem that all projective modules over $\K[x_1.\dots,x_n]$, with coefficients in a field $\K$, are free) (see~\cite[Example 4.8 (i) in Section 4B]{magurn}). There is also a construction by P.M. Cohn which by the current terminology is that of the Cohn path algebra associated to a vertex and two loops (see~\cite[Example 4.8 (iii)]{magurn})). 

In this short note, we investigate Serre's conjecture for the rings arising from Leavitt path algebras. We remark that Serre's conjecture has been looked at in some classes of other non-commutative rings such as Weyl algebras (see~\cite[VIII.8]{lam}). We first note that the Serre's conjecture holds for the Leavitt algebra $L_n, n \geq 2$,   that were originally defined and studied by William Leavitt.  Our investigation in this paper leads us to believe that for a unital Leavitt path algebra $L_\K(E)$ with coefficients in a field $\K$, if all finitely generated projective modules are free, then $L_\K(E)$ should be of the form $L_n$ (see Conjecture~\ref{conjfree}). 

The algebras $L_n$, for $n\geq 2$, constructed by Leavitt  can be obtained as Leavitt path algebras associated to graphs with one vertex and $n$ loops. We further set $L_0=\K$ and $L_1=\K[x,x^{-1}]$, which are Leavitt path algebras with one vertex and zero and one loop, respectively.

The precise conjecture we propose, if it holds, will distinguish classical Leavitt algebras from the rest of the large class of Leavitt path algebras. We use the notation $\mathbb N$ for natural numbers (including zero).

\begin{conjecture}\label{conjfree}
Suppose $E$ is a finite graph and $L_\K(E)$ its associated Leavitt path algebra. Then every finitely generated right projective $L_\K(E)$-module is free if and only if $L_\K(E)\cong L_n$, for some $n\in \mathbb N$. 
\end{conjecture}

The above conjecture is closely related to the algebraic Kirchberg-Phillips question which is said to be the most compelling unresolved question in the subject of Leavitt path algebras~\cite[\S 7.3.1]{AAS}. 

\begin{question}[Classification question]\label{questionpure}
Let $E$ and $F$ be finite graphs such that  $L_\K(E)$ and $L_\K(F)$ are purely infinite simple algebras. Is it true that  $L_\K(E)\cong L_\K(F)$ as $\K$-algebras if and only if there is an isomorphism $\phi: K_0(L_\K(E)) \rightarrow K_0(L_\K(F))$ such that 
$\phi([L_\K(E)])=[L_\K(F)]$?
\end{question}

The above question is in fact a theorem in the setting of graph $C^*$-algebras. Similar to the analytic case, the important test case is the comparison of the algebras $L_2$ and its Cuntz splice ${L_2}_{-}$.  
Consider the graph $E_2$ and the graph ${E_2}_{-}$ obtained by performing a Cuntz splice to $E_2$.

\begin{equation}\label{cuntzsplice}
E_2 \qquad \xymatrix{ 
		u \ar@(dl,ul) \ar@(dr,ur) 
	} 
	\qquad
	\qquad
	\quad
	\qquad {E_2}_{-} \qquad
	\xymatrix{ 
		u \ar@(l,u) \ar@(l,d) \ar@/^0.5em/[r] & v \ar@/^0.5em/[l] \ar@/^0.5em/[r] \ar@(ul,ur) & w \ar@/^0.5em/[l] \ar@(ur,dr)
	}
\end{equation}
\smallskip 

The Leavitt path algebra $L_\K(E_2)$ is the Leavitt algebra $L_2$, and the Leavitt path algebra $L_\K({E_2}_{-})$ is often denoted ${L_2}_{-}$.


\begin{question}[Splice question]\label{questionpure1}
Consider the Leavitt path algebras $L_2$ and ${L_2}_{-}$. Are these algebras isomorphic? 
\end{question}

In Theorem~\ref{calssthm} we show Conjecture~\ref{conjfree} sits between the two above questions. Namely a positive answer to Question~\ref{questionpure} implies Conjecture~\ref{conjfree}, which in return gives a positive answer to Question~\ref{questionpure1}. This means, if we could find a Leavitt path algebra which is not of the form $L_n$ whose finitely generated projective modules are all free, then the algebraic Kirchberg-Phillips Question~\ref{questionpure} fails to be true. 

In fact it would be interesting to determine the class of Leavitt path algebras whose projective modules are stably free as well as the class of Leavitt path algebras whose stably free modules are free (i.e., are Hermite rings~\cite{lam}).  Recall that for a unital ring $A$, a finitely generated projective $A$-module $P$ is \emph{stably free} if $P\oplus A^n\cong A^m$.
The latter question has been answered in \cite{nam}, when the Leavitt path algebras are assumed to have IBN, i.e., the Hermite rings were defined in the sense of Cohn, meaning they have IBN and stably free modules are free. The same questions can be asked for the graded version of these statements. Examples~\ref{exam1} and \ref{exam2}, show all these situations can arise in the large class of unital Leavitt path algebras and they remain to be systematically classified in this setting.


\section{Preliminary: Graphs, graph monoids and Leavitt path algebras}

Let $E=(E^{0}, E^{1}, r, s)$ be a directed graph, where $E^{0}$ and $E^{1}$ are
sets of vertices and edges and $s,r$ are source and range maps from $E^1$ to $E^0$, respectively.  A graph $E$ is said to be \emph{row-finite} if for each vertex $v\in E^{0}$,
$|s^{-1}(v)|< \infty$. If $s^{-1}(v)$ is empty we say $v$ is a sink. One can associate an algebra $L_\K(E)$, with coefficients in a field $\K$ to the graph $E$ which is  called the Leavitt path algebra. We refer the reader to the book of Abrams, Ara and Siles Molina \cite{AAS} for theory of Leavitt path algebras, its relation with the algebras defined by William Leavitt and all the standard terminologies we use here. 

We recall the notion of graph monoids and talented monoids as we heavily use them in this note.

\begin{deff}\label{def:graphmonoid}
 Let $E$ be a row-finite graph. 
 \begin{enumerate}
 
\item The \emph{graph monoid} $M_E$, is the commutative 
monoid generated by $\{v \mid v\in E^0\}$, subject to
\[v=\sum_{e\in s^{-1}(v)}r(e)\]
for every $v\in E^0$ that is not a sink.

\smallskip

\item The \emph{talented monoid} $T_E$, is the commutative 
monoid generated by $\{v(i) \mid v\in E^0, i\in \mathbb Z\}$, subject to
\[v(i)=\sum_{e\in s^{-1}(v)}r(e)(i+1)\]
for every $i \in \mathbb Z$ and every $v\in E^{0}$ that is not a sink. The additive group $\mathbb{Z}$ of integers acts on $T_E$ via monoid automorphisms by shifting indices: For each $n,i\in\mathbb{Z}$ and $v\in E^0$, define ${}^n v(i)=v(i+n)$, which extends to an action of $\mathbb{Z}$ on $T_E$. Throughout we will denote elements $v(0)$ in $T_E$ by $v$.

\end{enumerate}
\end{deff}

There is an explicit description of the congruence on the free abelian
monoid given by the defining relations of $M_{E}$ \cite[\S 3.6]{AAS}. Let $\mathbb F_E$ be the free abelian monoid on
the set $E^{0}$. The nonzero elements of $\mathbb F_E$ can be written uniquely up to
permutation as $\sum_{i=1}^{n}v_{i}$, where $v_{i}\in E^{0}$. Define a binary relation
$\rightarrow_{1}$ on $\mathbb F_E\setminus\{0\}$ by 
\begin{equation}\label{hfgtrgt655}
\sum_{i=1}^{n}v_{i}\longrightarrow_{1}\sum_{i\neq
j}v_{i}+\sum_{e\in s^{-1}(v_{j})}r(e),
\end{equation}
whenever $j\in \{1, \cdots, n\}$ and 
$v_{j}$ is not a sink. Let $\rightarrow$ be the transitive and reflexive closure of $\rightarrow_{1}$
on $\mathbb F_E\setminus\{0\}$ and $\sim$ the congruence on $\mathbb F_E$ generated by the relation $\rightarrow$.
Then $M_{E}=\mathbb F_E/\sim$.

Throughout $a \rightarrow_{1} b$ is called a
 \emph{transformation} of $a$ to $b$ in $\mathbb F_E\setminus\{0\}$.  The following lemma is crucial for our work and we frequently use it throughout the article. For the proof see~\cite[\S 3.6]{AAS}.

\begin{lemma}[The Confluence Lemma]\label{aralem6}
Let $E$ be a row-finite graph, $\mathbb F_E$ the free abelian monoid generated by $E^0$ and $M_E$ the graph monoid of $E$. For $a, b \in \mathbb F_E \backslash \{ 0 \}$, we have   $a=b$ in $M_E$ if and only if there is $c \in \mathbb F_E \backslash  \{0 \}$ such that 
$a \rightarrow c$ and $b\rightarrow c$. 

\end{lemma}

It was proved in~\cite[Theorem 3.2.5]{AAS}, using Bergman's machinery, that there is a monoid isomorphism  
\begin{align}\label{graphmonoid}
\phi:M_E &\longrightarrow \mathcal V (L_\K(E))\\
v &\longmapsto [vL_\K(E)],\notag
\end{align}
where $\mathcal V(L_\K(E))$ is the monoid of isomorphism classes of finitely generated right projective $L_\K(E)$-modules. Thus the group completion of $M_E$ retrieves the Grothendieck group $K_0(L_\K(E))$. When $E$ is finite, we denote $1_E:=\sum_{v\in E^{0}} v$ in $M_E$. In (\ref{graphmonoid}), we have $\phi(1_E)=[L_\K(E)]$. 

A similar result can be written for the graded setting (see~\cite{HazratLi}):  There is a $\mathbb Z$-monoid isomorphism  
\begin{align}\label{grgraphmonoid}
\phi:T_E &\longrightarrow \mathcal V^{\gr} (L_\K(E))\\
v(i) &\longmapsto [(vL_\K(E)(i)],\notag
\end{align}
where $\mathcal V^{\gr}(L_\K(E))$ is the monoid of graded isomorphism classes of finitely generated graded right projective $L_\K(E)$-modules. Thus the group completion of $T_E$ retrieves the graded Grothendieck group $K^{\gr}_0(L_\K(E))$ (\cite[\S 7.3.4]{AAS}).

\section{Results}

Let $A$ be a unital ring. It is easy to observe that any finitely generated (right) projective  $A$-module is free if and only if the monoid homomorphism 
$\mathbb N \rightarrow \mathcal V(A); 1 \mapsto [A]$, is surjective. In the case of Leavitt path algebras, combining this with the correspondence~(\ref{graphmonoid}), we obtain the following statement 
that will be used throughout. Recall that for a finite graph $E$, we denote $1_E:=\sum_{v\in E^{0}} v$ in $M_E$.

\begin{lemma}\label{thmfree}
Let $E$ be a finite graph and $M_E$ its associated graph monoid. Then any finitely generated projective $L_\K(E)$-module is free if and only if for any $v\in E^0$, $v=k 1_E$, for some $0 \not =  k\in \mathbb N$. 
\end{lemma}

\begin{example}\label{firstexm}
(1) Let $E$ be a graph with one vertex and $n$ loops, where $n\geq 1$. Then $L_\K(E)=L_n$. Since $M_E$ is generated by single vertex $v$, and $1_E=v$, clearly Lemma~\ref{thmfree} holds and thus any finitely generated projective module over $L_n$ is free. 

(2) Let $A=\M_d(L_n)$ be the $d \times d$-matrix algebra over $L_n$, where $n>1$ and $d>1$,  which is a Leavitt path algebra of a graph 
\begin{equation}
\xymatrix@=13pt{
E: u \ar@<1ex>[rr] \ar@{.>}[rr] \ar@<-1ex>[rr]&& v\ar@(u,r) \ar@(ur,dr)  \ar@{.}@(r,d) \ar@(dr,dl),
}
\end{equation}

\medskip 
\noindent consisting of a $d-1$ edges connecting the vertex $u$ to vertex $v$  which itself has  $n$ loops. If all finitely generated projective modules over $A=L_\K(E)$ are free, then by Lemma~\ref{thmfree}, $v=k 1_E$ in $M_E$, where $k>0$. Thus in $M_E$, $v=k(u+v)=k((d-1)v+v)=kdv$, for some $0\not = k\in \mathbb N$.  Applying the Confluence Lemma~\ref{aralem6} to this equality in $M_E$, we obtain $v+i(n-1)v=kdv+j(n-1)v$ in the free monoid $\mathbb F_E$, for some $i> 0$ and $j\geq 0$. This gives that $1+i(n-1)=kd+j(n-1)$. This immediately implies $\gcd(d,n-1)=1$. Thus by \cite[Corollary 6.3.44]{AAS}, $L_\K(E)\cong L_n$ which could be considered as the first test of our Conjecture~\ref{conjfree}. 
\end{example}

\begin{rmk}  By \cite[Corollary 18.36]{lam2} if the unital rings $A$ and $B$ are Morita equivalent and any finitely generated projective $A$-module is free then $B\cong \M_d(A)$, for some $d\geq 1$.  Combining this with  Example~\ref{firstexm}(2), it is enough to prove a weaker version of our Conjecture~\ref{conjfree}: Given a finite graph $E$, if every finitely generated right $L_\K(E)$-module is free, then $L_\K(E)$ is Morita equivalent to the Leavitt ring $L_n$ for some integer $n \geq  1$. Since every finitely generated right $L_\K(E)$-module is free, then \cite[Corollary 18.36]{lam2} implies that $L_\K(E) \cong \M_d(L_n)$ for some integer $d \geq 1$. By Example~\ref{firstexm}(2), we then have that $L_\K(E) \cong L_n$.
\end{rmk}

Recall that a unital ring $A$ is called \emph{purely infinite simple} if $A$ is not a division ring, and for any nonzero $x \in  A$,  there exist $a,b\in A$ such that $axb=1$~\cite[p. 214]{AAS}. For the case of Leavitt path algebras, a complete characterisation of purely infinite simple algebras based on the geometry of associated graphs was among the first published results in this theory~\cite[Theorem 3.1.10]{AAS}.


\begin{thm}\label{thmpurelyinf}
Let $E$ be a finite graph and $L_\K(E)$ its associated Leavitt path algebra. If any  finitely generated projective $L_\K(E)$-module is free then $L_\K(E)$ is one of the following:
\begin{enumerate}[\upshape(i)]
\item $L_\K(E)\cong \K$;

\smallskip

\item $L_\K(E)\cong \K[x,x^{-1}]$;

\smallskip

\item $L_\K(E)$ is a purely infinite simple algebra with $\big(K_0(L_\K(E)), [L_\K(E)]\big) \cong (\mathbb Z/n\mathbb Z, 1)$, for some $0\not = n\in \mathbb N$.
\end{enumerate}

\end{thm}
\begin{proof}
If the algebra is of the form of (i) or (ii), which is a Leavitt path algebra associated to a vertex or a single loop, then it is a PID ring. It is a classical result that 
finitely generated projective modules over PID are free.  Furthermore if the graph $E$ consists of one vertex, and thus $n$ loops, where $n>1$, then $L_\K(E)\cong L_n$, for some $n$ and by Example~\ref{firstexm}(1)   all finitely generated projective modules are free modules. On the other hand we know $L_n, n>1$ is a purely infinite simple algebra~\cite[Remark~3.1.9]{AAS}. 

Suppose now that the graph $E$ consists of more than one vertex.  Let $I$ be a non-zero order ideal of the graph monoid $M_E$.  Then for some $u \in E^0$, 
 $u\in I$. By Lemma~\ref{thmfree}, $k1_E=u \in I$, where $0 \not = k\in \mathbb N$. Thus for any $v\in E^0$, $v\leq k1_E \in I$. Since $I$ is an order ideal, $v\in I$. It follows that $I=M_E$. Therefore $M_E$ has no non-trivial order ideals. Since the order ideals of $M_E$ are in one-to-one correspondence with the hereditary and saturated subsets of $E^0$ (\cite[Proposition 3.6.9]{AAS}), and also with the graded ideals of $L_\K(E)$ (\cite[Theorem~2.5.9]{AAS}), it follows that $L_\K(E)$ is graded simple and the only hereditary and saturated subsets are $\emptyset$ and $E^0$. Next, suppose that a vertex $v\in E^0$ is a sink or a vertex on a cycle without exit. Then  the equality $v =  k1_E$, $k \geq 1$,  is not possible, as either there is no transformation for $v$ (if it is a sink) or any transformation will only generate one vertex, whereas $k1_E=k(\sum_{v\in E^0}v)$ consists of more than one vertex. 
Thus any cycle in the graph has an exit. By~\cite[Theorem 3.1.10]{AAS} this implies that $L_\K(E)$ is purely infinite simple. 

Consider the map 
\begin{align*}
\psi: \mathbb N &\longrightarrow M_E\\
1 & \longmapsto 1_E.
\end{align*}
Since by Lemma~\ref{thmfree}, any vertex $u\in M_E$ is $k_u$ copies of $1_E$, where $k_u>0$, this map is epimorphism. Composing with the isomorphism (\ref{graphmonoid})
and  passing to the group completion,   we obtain an epimorphism (call it $\phi$ again) 
\begin{align}\label{mapsio3}
\phi:\mathbb Z &\longrightarrow K_0(L_\K(E))\\
 1&\longmapsto  [L_\K(E)].\notag
\end{align} 
Since 
\[\phi(1)= [L_\K(E)] =\sum_{u\in E^0} [uL_\K(E)]=\sum_{u\in E^0} k_u [L_\K(E)]=\phi(\sum_{u\in E^0} k_u),\]
and $\sum_{u\in E^0} k_u>1$ as $E^0$ consists of more than one vertex, it follows that the map $\phi$ of (\ref{mapsio3}) does have a non-trivial kernel. 
Hence we obtain an isomorphism $\big(K_0(L_\K(E)), [L_K(E)]\big) \cong (\mathbb Z/n\mathbb Z, 1)$, for some $0\not = n\in \mathbb N$. 
\end{proof}

We are in a position to show Serre's conjecture sits in between the two unanswered questions. 

\begin{thm}\label{calssthm}
If Question~\ref{questionpure} has a positive answer then Conjecture~\ref{conjfree} is true. 
If Conjecture~\ref{conjfree} is true then Question~\ref{questionpure1} has a positive answer.
\end{thm}
\begin{proof}
Suppose Question~\ref{questionpure} has a positive answer and assume that every  finitely generated projective module over $L_\K(E)$ is free. If $L_\K(E)\cong L_n$ for some $n\in \mathbb N$, then  we are done. Thus, suppose that $E$ consists of more than one vertex.  By Theorem~\ref{thmpurelyinf}, $L_\K(E)$ is purely infinite simple with
$\big(K_0(L_\K(E)), [L_\K(E)]\big) \cong (\mathbb Z/n\mathbb Z, 1)$, for some $0\not = n\in \mathbb N$.  Since $\big(K_0(L_{n+1}),[L_{n+1}]\big)\cong (\mathbb Z/n\mathbb Z, 1)$, comparing these two data and  since Question~\ref{questionpure} has a positive answer, it follows that $L_\K(E)\cong L_{n+1}$, for some $n > 0$. 

On the other hand, suppose the Leavitt path algebras for which Serre's conjecture holds are of the form $L_n$ for some  $n \geq 0$. We check directly from the relations that  every  finitely generated projective module over ${L_2}_{-}$  is free. Consider the graph ${E_2}_{-}$ from (\ref{cuntzsplice}) that is associated to ${L_2}_{-}$.  The graph monoid $M_{{E_2}_{-}}$ is generated by the set $\{u,v,w\}$ subject to relations 
\begin{align*}
& (1) \, \, \, u = u+u+v \\
& (2) \, \, \,  v=u+v+w \\   
& (3) \, \, \, w=v+w.\\ 
& \text{Combing (2) and (3) we also have:} \\
 &(4) \, \, \,  v=u+w.
\end{align*}

Set $1_{{E_2}_{-}}=u+v+w$.  Then by (2) we have \[v=1_{{E_2}_{-}}.\] On the other hand 
\[w\stackrel{(3)}{=}v+w \stackrel{(2)}{=} u+v+w+w\stackrel{(1)}{=}u+u+v+v+w+w=u+\underbrace{(u+v+w)}_{(2)}+\underbrace{(v+w)}_{(3)}=u+v+w=1_{{E_2}_{-}}.\]
Hence 
\[u\stackrel{(1)}{=}u+u+v\stackrel{(4)}{=}u+u+u+w\stackrel{(3)}{=}\underbrace{u+w}_{(4)}+\underbrace{u+u+v}_{(1)}=u+v\stackrel{(2)}{=}\underbrace{u+u+v}_{(1)}+w=u+w\stackrel{(3)}{=}1_{{E_2}_{-}}.\]
Thus, by Lemma~\ref{thmfree}, every  finitely generated projective module of ${L_2}_{-}$ is free. Therefore, by assumption,  ${L_2}_{-}\cong L_n$, for some $ n\geq 0$. Since in $M_{{E_2}_{-}}$, we have $u=v=w=1_{{E_2}_{-}}$ and relation (3) gives that $u=u+u$, we therefore have 
$M_{{E_2}_{-}}\cong M_{E_2}$.  By comparing the monoids, we get ${L_2}_{-} \cong L_2$. 
\end{proof}

\begin{example} \label{exam1}
(1)  Let $E$ be the graph 
 
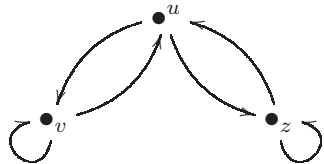
\[ \quad {
\def\labelstyle{\displaystyle}
\xymatrix{ {} & \bullet^{u}  \ar@/^{-10pt}/[rd] \ar@/^{-10pt}/ [ld] &  {} \\
\bullet_{v} \ar@(d,l) \ar@/^{-10pt}/[ru] &  & \bullet_{z} \ar@(d,r)
\ar@/^{-10pt}/ [lu] \\
}}
\]
\smallskip

Then $K_0(L_\K(E))\cong \mathbb Z$ with $[L_K(E)]=0$. The algebra $L_\K(E)$ is a purely infinite simple ring, but the calculation of Theorem~\ref{calssthm} shows not all finitely generated projective modules are free over $L_\K(E)$ (otherwise $K_0(L_\K(E))$ would be torsion). In fact, not all finitely generated projective modules over $L_\K(E)$ are \emph{even} stably free.  Indeed if  for a finitely generated projective $A$-module $P$, we have $[P \oplus A^n]=[A^m]$, where in our setting $[A]=[L_\K(E)]=0$,  then $[P]=0$. So if this happens for every finitely generated projective module $P$, we would have $K_0(L_\K(E))=0$, which is not the case. 

(2) 
It is easy to see that for a unital ring $A$, the group homomorphism $\mathbb Z \rightarrow K_0(A); 1\mapsto [A]$ is surjective if and only if all finitely generated projective modules are stably free. In particular if $K_0(A)=\mathbb Z$ with the order unit $[A]=1$, then all finitely generated projective modules are stably free. If we construct a noncommutative, non-purely infinite simple Leavitt path algebra $L_\K(E)$ with $K_0(L_\K(E))\cong \mathbb Z$ as above, then Theorem~\ref{thmpurelyinf} implies that $A$ has a stably free module which is not free. 

Let $E$ be the graph 

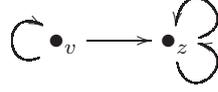
\[ \quad {
\def\labelstyle{\displaystyle}
\xymatrix{ 
\bullet_{v} \ar@(ld,lu) \ar[r] &  \bullet_{z} \ar@(r,u) \ar@(d,r)
}}
\]
\smallskip

An easy application of the Confluence Lemma~\ref{aralem6} shows that $L_\K(E)$ has IBN. Since $L_\K(E)$ is not purely infinite simple, by Theorem~\ref{thmpurelyinf}, not all finitely generated projective modules are free.  However the following calculation in $M_E$ shows all finitely generated projective modules over $L_\K(E)$ are stably free. 
\begin{align*}
&v=v+z=1_E\\
&z+1_E=z+v+z=v+z=1_E. 
\end{align*}
Passing to the monoid $\mathcal V(L_\K(E))$ via the isomorphism~\ref{graphmonoid}, we obtain $vL_\K(E)\cong L_\K(E)$ and $zL_\K(E) \oplus L_\K(E) \cong L_\K(E)$.  
 We know any finitely generated projective $L_\K(E)$-module is generated by $vL_\K(E)$ and $zL_\K(E)$. Since $vL_\K(E)$ is free, the stably free right ideal $zL_\K(E)$ can't be free (otherwise all finitely generated projective modules would be free). Note that $K_0(L_\K(E))=\mathbb Z$ with $[L_\K(E)]=1$. 

\end{example}

\begin{example}\label{exam2}
The graded version of Serre's conjecture is too weak to imply that the algebra is graded isomorphic to $L_n$ for some $n \in \mathbb N$.
The following example gives an algebra $L_\K(E)$ where all  finitely generated \emph{graded}  projective $L_\K(E)$-modules are graded free, however $L_\K(E)$ is not graded isomorphic to $L_n$, for some $n \in \mathbb N$. 
Consider the Leavitt path algebra $L_\K(E)$ associated to the graph 

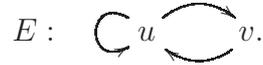
\begin{equation*}
\xymatrix{
E: & u  \ar@(lu,ld)\ar@/^0.9pc/[r] & v. \ar@/^0.9pc/[l]
}
\end{equation*}
\smallskip

Employing the $\mathbb Z$-monoid isomorphism (\ref{grgraphmonoid}), and the relations in the talented monoid $T_E$ (Definition~\ref{def:graphmonoid}(2)), we have 
\[[uL_\K(E)]=\phi(u)=\phi(u(1)+v(1))=[uL_\K(E)(1)]+[vL_\K(E)(1)]=[L_\K(E)(1)].\] In particular we have $\phi(u(1))={}^1\phi(u)={}^1[uL_\K(E)(1)]=[uL_\K(E)(2)]$. Thus 
\[[vL_\K(E)]= \phi(v)=\phi(u(1))=[L_\K(E)(2))].\]   By isomorphism (\ref{grgraphmonoid}) in our setting, all finitely generated graded projective modules over $L_\K(E)$ are generated by some shifts of $uL_\K(E)$ and $vL_K(E)$. Thus the above computations show that  they are all graded free. However $T_E\cong\Big \{ (m,n)\in \mathbb Z \oplus \mathbb Z  \, \, \Big \vert \, \,   \frac{1+\sqrt{5}}{2} m +n \geq 0 \Big \} $,  whereas $\mathcal V^{\gr}(L_n)\cong \mathbb N[1/n]$ (see~\cite[Example 2.5]{HazratLi}) showing that $L_\K(E)$ can't be graded isomorphic to some Leavitt algebra $L_n$, $n\in \mathbb N$.

\end{example}

\begin{rmk}[Serre's conjecture for graph $C^*$-algebras]\label{serrecstar}

 For graph $C^*$-algebras where the concepts of projective and free modules are not formally considered in the
definition of $\mathcal V$-monoids, motivated by Lemma~\ref{thmfree}, we take a monoidal approach to address this question.

For a finite graph $E$, we say that the \emph{Serre's conjecture holds} in $C^*(E)$ if for each $v\in E^{0}$%
, there is a positive integer $k$ such that $v=k1_{E}$ in the graph monoid $M_{E}$.

Given that Question~\ref{questionpure} has a positive answer in the setting of graph $C^*$-algebras (the Kirchberg-Phillips Theorem)~\cite[Theorem~6.3.1, Remark 6.3.2]{AAS}, writing Theorem~\ref{calssthm} mutatis mutandis in the $C^*$-setting, shows that 
for a finite graph $E$, the graph $C^*$-algebra $C^{\ast }(E)$ has the Serre's conjecture property if and only if, for some $n>0$, 
$C^{\ast }(E)$ is isomorphic to the Cuntz algebra  $\mathcal{O}_{n}$ (note that we include also $n=0$ and $n=1$ here). 

We note that a positive answer to the Kirchberg-Phillips question for $C^*$-algebras was obtained by using sophisticated analytical tools and topological methods, using approximations and limits  applicable to a normed vector space~\cite{Ror}. Such tools are not available for Leavitt path algebras making the algebraic Kirchberg-Phillips question an open question.
\end{rmk}

\section{Acknowledgements}
A part of this work was done when the first author was an Alexander von Humboldt Fellow at the University of M\"unster in the winter of 2021. He would like to thank both institutions for an excellent hospitality. The authors would like to express their appreciations to the referees for their valuable comments. We would also like to thank Mark Tomforde for his explanations about the Kirchberg-Phillips Theorem.


\begin{thebibliography}{9999}



\bibitem{AAS} G. Abrams, P. Ara, M. Siles Molina, Leavitt path algebras. Lecture Notes in Mathematics, vol. 2191, Springer Verlag, 2017.


\bibitem{nam} G. Abrams, T. G. Nam and N. T. Phuc, {\it Leavitt path algebras having unbounded generating number}, J. Pure Appl. Algebra $\mathbf{221}$ (2017), 1322--1343



\bibitem{HazratLi} R. Hazrat, H. Li, {\it The talented monoid of a Leavitt path algebra}, J.  Algebra $\mathbf{547}$ (2020), 430--455.


\bibitem{lam} T.Y. Lam, Serre's problem on projective modules. Springer Monographs in Mathematics. Springer-Verlag, Berlin, 2006. 

\bibitem{lam2} T.Y. Lam, Lectures on Modules and Rings, Graduate Texts in Mathematics vol. 189, Springer-Verlag, New York, (1999).

\bibitem{magurn} B. Magurn, An algebraic introduction to K-theory, \emph{Encyclopedia of Mathematics and its Applications}, 87. Cambridge University Press, Cambridge, 2002. 


\bibitem{nam2}  T. G. Nam and N. T. Phuc, {\it The structure of Leavitt path algebras and the invariant basis number property}, J. Pure Appl. Algebra 223 (2019), no. 11, 4827--4856.

\bibitem{Ror} M. R\o{}rdam, Classification of nuclear, simple $C^*$-algebras. \emph{Classification of nuclear $C^*$-algebras. Entropy in operator algebras}, 1--145, Encyclopaedia Math. Sci., 126, Oper. Alg. Non-commut. Geom., 7, Springer, Berlin, 2002.

\end{thebibliography}
\end{document}